\documentclass[12pt]{amsart}

\usepackage{amssymb,amsmath}
\usepackage{amssymb,latexsym}
\usepackage{amsfonts}
\usepackage{amssymb}
\usepackage{stmaryrd}
\usepackage{longtable}
\usepackage{hyperref}
\usepackage{amsmath, geometry, amssymb} 
\numberwithin{equation}{section}
\newtheorem{theorem}{Theorem}[section]

\newcommand{\sqr}[2]{{\vcenter{\vbox{\hrule height#2pt
                \hbox{\vrule width#2pt height#1pt \kern#1pt
                \vrule width#2pt}\hrule height#2pt}}}}

\newcommand{\beq}{\begin{equation}}
\newcommand{\eeq}{\end{equation}}
\newcommand{\beqar}{\begin{eqnarray}}
\newcommand{\eeqar}{\end{eqnarray}}
\def\beqars{\begin{eqnarray*}}
\def\eeqars{\end{eqnarray*}}

\DeclareMathOperator{\lcm}{lcm}

\newcommand{\smod}[1]{\hspace{-1mm} \pmod{#1}}

\def \ds{\displaystyle}

\newcommand{\nn}{\mathbb{N}}
\newcommand{\zz}{\mathbb{Z}}
\newcommand{\qq}{\mathbb{Q}}
\newcommand{\cc}{\mathbb{C}}

\newcommand{\hh}{\mathbb{H}}

\allowdisplaybreaks

\begin{document}

\markboth{Zafer Selcuk Aygin}
{Convolution Sums}

\title{Eisenstein Series and Convolution Sums}

\author{Zafer Selcuk Aygin}

\maketitle

\begin{abstract}
We compute Fourier series expansions of weight $2$ and weight $4$ Eisenstein series at various cusps. Then we use results of these computations to give formulas for the convolution sums $ \sum_{a+p b=n}\sigma(a)\sigma(b)$, $ \sum_{p_1a+p_2 b=n}\sigma(a)\sigma(b)$ and $ \sum_{a+p_1 p_2 b=n}\sigma(a)\sigma(b)$ where $p, p_1, p_2$ are primes.\\


\end{abstract}


\section{Introduction}
Let $\nn$, $\nn_0$, $\zz$, $\qq$, $\cc$ and $\hh$ denote the sets of positive integers, non-negative integers, integers, rational numbers, complex numbers and the upper half plane, respectively. Throughout the paper we let $z \in \hh$ and  $q=e^{2 \pi i z}$. Let $N\in\nn$. Let $\Gamma_0(N)$ be  the modular subgroup defined by
\beqars
\Gamma_0(N) = \left\{ \left(
\begin{array}{cc}
a & b \\
c & d
\end{array}
\right)  \mid  a,b,c,d\in \zz ,~ ad-bc = 1,~c \equiv 0 \smod {N}
\right\} .
\eeqars 
We note that the matrices
\beqar
S=\left(
\begin{array}{cc}
	-1 & 0 \\
	0 & 1
\end{array}
\right) \mbox{ and } T=\left(
\begin{array}{cc}
	1 & 1 \\
	0 & 1
\end{array}
\right). \label{eq:14}
\eeqar
are the generators of $\Gamma_0(1) ~(=SL_2(\zz))$. An element {\renewcommand*{\arraystretch}{0.5} $M= \left({\begin{array}{cc}
		a & b \\       c & d \\      \end{array} } \right) \in \Gamma_0(1)$} acts on $\hh \cup \qq \cup \{ \infty \}$ by
\beqars
\ds M(z)=\left\{\begin{array}{ll}
	\frac{az+b}{cz+d} & \mbox{ if } z \neq \infty, \\
	\frac{a}{c} & \mbox{ if } z= \infty .
\end{array} \right.
\eeqars
Let $k \in \nn$. We write $M_k(\Gamma_0(N))$ to denote the space of modular forms of weight $k$ for  $\Gamma_0(N)$, and 
$E_k (\Gamma_0(N))$ and $S_k(\Gamma_0(N))$ to denote the subspaces of Eisenstein forms and cusp forms 
of  $M_k(\Gamma_0(N))$, respectively. 
It is known  (see for example \cite[p. 83]{stein} and \cite{Serre}) that
\beqar
M_k (\Gamma_0(N)) = E_k (\Gamma_0(N)) \oplus S_k(\Gamma_0(N)). \label{eq:10}
\eeqar

Let $A_c=\left(
\begin{array}{cc}
-1 & 0 \\
c & -1
\end{array}
\right)$, then the Fourier series expansion of $f(z) \in M_k(\Gamma_0(N))$ at the cusp $\ds \frac{1}{c}\in \qq \cup \{\infty \}$ is given by the Fourier series expansion of $f(A_c^{-1}z)$ at the cusp $\infty$, see \cite[pg. 35]{Kohler}. If the Fourier series expansion of $f(z)$ at the cusp $\ds \frac{1}{c} $ is given by the infinite sum
\beqars
(cz+1)^k \sum_{n\geq 0} a_n e^{2 \pi i n z_c},
\eeqars
then we use the notation $[n]_cf(z) = a_n$. It is to use $[n]$, instead of $[n]_0$, at the cusp $\infty ~(=1/0)$.

For $k \in \nn_0$ and $n \in \nn$ we define the sum of divisors function by
\beqars
\sigma_k(n)=\sum_{d \mid n}d^k
\eeqars
where $d$ runs through positive divisors of $n$. If $n \notin \nn$ we set $\sigma(n)=0$. When $k=1$, we write $\sigma(n)$ instead of $\sigma_1(n)$. We define the weight $2$ and the weight $4$ Eisenstein series by
\begin{align}
& E_2(z)=1-24\sum_{n \geq 1} \sigma(n) q^n, \label{eq:5} \\
& E_4(z)=1+240\sum_{n \geq 1} \sigma_3(n) q^n, \label{eq:6}
\end{align}
respectively. Let $d \in \nn$, we denote $E_{2}(dz)$ by $E_{(2,d)}(z)$ and $E_{4}(dz)$ by $E_{(4,d)}(z)$. For convenience let us define $L_d(z)=E_2(z)-dE_{(2,d)}(z)$. It is known that 
\beqar
&& L_d(z) \in E_2(\Gamma_0(d)), \label{eq:3}\\
&& E_4(dz) \in E_4(\Gamma_0(d)),   \label{eq:4}
\eeqar
see \cite[Theorem 5.8]{stein}.

Let $r, s \in \nn$. We define
\beqars
W(r,s;n)=\sum_{\substack{ a, b \in \nn,\\ar+bs=n}}\sigma(a)\sigma(b)
\eeqars
to be the convolution of sum of divisors function. The formula 
\begin{align*}
W(1,1;n)=\frac{5}{12} \sigma_3(n)+\frac{1-6n}{12} \sigma(n)
\end{align*}
appears in the works of Besge, Glaisher and Ramanujan, see \cite{8besge,12glaisher,19ramanujanocaf}, respectively. Since then $W(r,s;n)$ has been calculated for various $(r,s)$. In the table below we grouped the known results according to nature of $r$ and $s$, together with the references.
{\small\begin{center}
	\begin{longtable}{l | l | l}
		\caption{Previously known formulas for convolution sums $W(r,s;n)$ } \label{table:6}
		\endhead
		Nature of $r$ and $s$ & $(r,s)$ & References   \\
		\hline
		$(r,s)=(1,1)$ & $ (1,1) $ & \cite{8besge,12glaisher,19ramanujanocaf} \\
		$(r,s)=(1,p)$ & $ (1,2),(1,3),(1,5),(1,7), (1, 11), (1, 13), (1, 23) $ & \cite{ chancooper, 10coopertoh, 13huard,15lemire}  \\
		$ (r,s)=(p_1,p_2) $  & $(2,3), (2, 5), (3, 5) $  & \cite{7alacawilliams,11cooperye,18ramaksahu} \\
		$ (r,s)=(1,p_1p_2) $  &  $ (1,6), (1, 10), (1, 14), (1, 15)$  & \cite{7alacawilliams,18ramaksahu,20royer} \\
		$ (r,s)=(1,p^i) $  & $ (1,4),(1,8),(1,9),  (1, 16), (1, 25), (1,27), (1,32) $  & \cite{6aaw,alacakesici,13huard,21williams,22williams2,23xia} \\
		$ (r,s)=(1,p_1^ip_2^j) $  & $ (1, 12),(1, 18), (1, 20), (1, 24), (1, 36) $  & \cite{3aaw,5aaw,11cooperye} \\
		$ (r,s)=(p_1^i,p_2^j) $  & $  (3, 4), (2, 9), (4, 5), (3, 8), (4, 9) $  & \cite{4aaw,5aaw,11cooperye}
	\end{longtable}
\end{center}}
In this paper we give formulas for $W(1,p;n)$, $W(p_1,p_2;n)$ and $W(1,p_1p_2;n)$. In Theorem \ref{th:5} we give the precise expression for the Eisenstein part of the formulas for $W(1,p;n)$, $W(p_1,p_2;n)$ and $W(1,p_1p_2;n)$ for all primes $p, p_1, p_2$ such that $p_1\neq p_2$. These formulas generalize the previously known formulas in rows 2--4 referenced in Table \ref{table:6}. Our method proves the existence of the cusp forms satisfying the formulas but fails to provide a general expression for the cusp part. Later on in the paper we will give the cusp part of the formulas for $W(r,s;n)$ with $(r,s)=(1,2),$ $(1,3),$ $(1,5),$ $(1,7),$ $(1,11),$ $(2,3),$ $(2,5),$ $(2,7),$ $(3,5),$ $(1,6),$ $(1,10),$ $(1,14),$ $(1,15)$ in terms of linear combinations of eta quotients. 

The structure of the paper is as follows. In the next section we give the statements of main results. In the third section we derive Fourier series expansions of weight $2$ and weight $4$ Eisenstein series at various cusps. Then, in Section \ref{sec:4}, we use the first terms of Fourier series expansions of Eisenstein series at certain cusps to prove the main results. In Section \ref{sec:5} we give a description for the cusp forms $C_{(r,s)}(z)$ with $(r,s)=(1,2),$ $(1,3),$ $(1,5),$ $(1,7),$ $(1,11),$ $(2,3),$ $(2,5),$ $(2,7),$ $(3,5),$ $(1,6),$ $(1,10),$ $(1,14),$ $(1,15)$ in terms of eta quotients. Finally in Section \ref{sec:6}, for the interested reader, we describe how to extend the results of this paper to give formulas for $W(r,s;n)$ with $\lcm(r,s)$ a square-free, two times a square-free number or four times a square-free number. 

\section{Main Results}
In this section we state the main results.

\begin{theorem}\label{th:1}
Let $p$ be a prime. Then there exists a cusp form $C_{(1,p)}(z) \in S_4(\Gamma_0(p))$ such that 
\begin{align}
\ds & (L_p(z))^2=\frac{(p-1)^2}{p^2+1}E_4(z)+\frac{p^2(p-1)^2}{p^2+1}E_4(pz)+C_{(1,p)}(z). \label{eq:1} 
\end{align}
Let $p_1, p_2$ be primes such that $p_1\neq p_2$. Then there exists cusp forms $C_{(p_1,p_2)}(z),$ $ C_{(1,p_1p_2)}(z) \in S_4(\Gamma_0(p_1p_2))$ such that 
\begin{align}
\ds  L_{p_1}(z)L_{p_2}(z)=&\frac{(p_1^2-p_1+1)(p_2^2-p_2+1)}{(p_1^2+1)(p_2^2+1)}E_4(z)-\frac{p_1^3(p_2^2-p_2+1)}{(p_1^2+1)(p_2^2+1)}E_4(p_1z) \label{eq:2}\\
\ds &-\frac{p_2^3(p_1^2-p_1+1)}{(p_1^2+1)(p_2^2+1)}E_4(p_2z)+\frac{p_1^3p_2^3}{(p_1^2+1)(p_2^2+1)}E_4(p_1p_2z)+C_{(p_1,p_2)}(z). \nonumber
\end{align}
\begin{align}
	\ds  (L_{p_1p_2}(z))^2=&\left(1-\frac{2p_1 p_2}{(p_1^2+1)(p_2^2+1)}\right)E_4(z)-\frac{2p_1^3 p_2}{(p_1^2+1)(p_2^2+1)}E_4(p_1z) \label{eq:16}\\
	\ds &-\frac{2p_1 p_2^3}{(p_1^2+1)(p_2^2+1)}E_4(p_2z)+\left(p_1^2p_2^2- \frac{2p_1^3p_2^3}{(p_1^2+1)(p_2^2+1)} \right)E_4(p_1p_2z)\nonumber \\
	&+C_{(1,p_1p_2)}(z). \nonumber
\end{align}

\end{theorem}

We compare the coefficients of $q^n$ on both sides of equations (\ref{eq:1})--(\ref{eq:16}) to get the following theorem.
\begin{theorem} \label{th:5} Let $p, p_1, p_2$ be primes such that $p_1 \neq p_2$, then for all $n\in \nn $ we have
\beqars
& W(1,p;n)=& \frac{5}{12(p^2+1)}  \sigma_3(n)+  \frac{5p^2}{12(p^2+1)} \sigma_3(n/p) + \frac{p-6n}{24p}\sigma(n) \\
&& +\frac{ 1-6 {n} }{24} \sigma(n/p)  -\frac{1}{1152p} [n] C_{(1,p)}(z),\\
& W(p_1,p_2;n)= & \frac{5}{12 (p_1^2+1)(p_2^2+1)} \sigma_3(n)+\frac{5p_1^2}{12 (p_1^2+1)(p_2^2+1)}  \sigma_3(n/p_1)\\
&&+\frac{5p_2^2}{12 (p_1^2+1)(p_2^2+1)} \sigma_3(n/p_2) +\frac{5p_1^2p_2^2}{12 (p_1^2+1)(p_2^2+1)} \sigma_3(n/p_1p_2)\\
&&+\frac{ p_2 - 6 n}{24 p_2} \sigma(n/p_1)  +\frac{p_1 - 6n  }{24p_1} \sigma(n/p_2) - \frac{1}{1152p_1p_2}[n] C_{(p_1,p_2)}(z)\\
& W(1,p_1p_2;n)= & \frac{5}{12 (p_1^2+1)(p_2^2+1)} \sigma_3(n)+\frac{5p_1^2}{12 (p_1^2+1)(p_2^2+1)}  \sigma_3(n/p_1)\\
&&+\frac{5p_2^2}{12 (p_1^2+1)(p_2^2+1)} \sigma_3(n/p_2) +\frac{5p_1^2p_2^2}{12 (p_1^2+1)(p_2^2+1)} \sigma_3(n/p_1p_2)\\
&&+\frac{ p_1p_2 - 6 n}{24 p_1 p_2} \sigma(n)  +\frac{1 - 6n  }{24} \sigma(n/p_1p_2) - \frac{1}{1152p_1p_2}[n] C_{(1,p_1p_2)}(z),
\eeqars
where $C_{(1,p)}(z)$, $C_{(1,p_1p_2)}(z)$ and $C_{(p_1,p_2)}(z)$ are the cusp forms from Theorem \ref{th:1}. 
\end{theorem}

Note that Chan and Cooper in \cite{chancooper} gave the equation for $W(1,p;n)$, valid for $p=3,7,11,23$. The closed form they gave for Eisenstein part of the formula is the same as Eisenstein part of the formula for $W(1,p;n)$ given in the previous theorem.  

\section{Fourier series expansions of weight $2$ and weight $4$ Eisenstein series at various cusps}
In this section we find the Fourier series expansion of weight $2$ and weight $4$ Eisenstein series at various cusps. The results of this section will be used to prove Theorem \ref{th:1}.
The transformation formula for $E_2(z)$ under the matrices $T$ and $S$ are given by
\beqar
&& E_2(Tz)=E_2(z+1)=E_2(z), \label{eq:15}\\
&& E_2(Sz)=E_2(-1/z)=z^2\left( E_2(z)- \frac{1}{2 \pi i z}  \right), \label{eq:12}
\eeqar
see \cite[Prop. 2.9]{Kilford}. 
\begin{theorem} \label{th:2}
Let $1<t \in \nn$. The Fourier series expansion of $L_t(z)$ at cusp $1 \in \qq$ is given by
\begin{align}
L_t(A_1^{-1}z) = (z+1)^2 \left( E_2(z) - \frac{1}{t} E_2\left( \frac{z+1}{t} \right) \right). \label{eq:17}
\end{align}
Let $p_1,p_2$ be prime. The Fourier series expansions of $L_{p_1p_2}(z)$ at cusps $1/{p_1}, 1/{p_2} \in \qq$ are given by
\begin{align}
& L_{p_1p_2}(A_{p_1}^{-1}z) = \left({p_1 z +1 }\right)^2 \left(  E_2(z) -\frac{p_1}{p_2} E_{2} \left(\frac{p_1z+1}{p_2} \right) \right) \mbox{ and }\label{eq:18}\\
& L_{p_1p_2}(A_{p_2}^{-1}z) = \left({p_2 z +1 }\right)^2 \left(  E_2(z) -\frac{p_2}{p_1} E_{2} \left(\frac{p_2z+1}{p_1} \right) \right),\label{eq:19}
\end{align}
respectively.
\end{theorem}
\begin{proof} Let $t \in \nn$ and the matrices $S$ and $T$ are as in (\ref{eq:14}). We have
\begin{align}
E_{(2,t)}(A_1^{-1}z)&=E_2\left(S^2T^{t}S\left(\frac{z+1}{t}\right)\right) \nonumber\\
& =E_2\left(T^{t}S\left(\frac{z+1}{t}\right)\right) \nonumber\\
& =E_2\left(S\left(\frac{z+1}{t}\right)\right)\nonumber\\
& =\left(\frac{z+1}{t}\right)^2\left( E_2\left(\frac{z+1}{t}\right)- \frac{t}{2 \pi i \left({z+1}\right)}  \right). \label{eq:13}
\end{align}
where in the first, second and third steps we use $E_2(S^{2}(z))=E_2(z)$, $E_2(T^{t}(z))=E_2(z)$ and (\ref{eq:12}), respectively. Let $1<t \in \nn$. Then we use (\ref{eq:13}) to get
\begin{align*}
 L_t(A_1^{-1}z)&=E_{2}(A_1^{-1}z)-t E_{(2,t)}(A_1^{-1}z)\\
&= \left({z+1}\right)^2\left( E_2\left({z+1}\right)- \frac{1}{2 \pi i \left({z+1}\right)}  \right) \\
& ~~~~- t\left(\frac{z+1}{t}\right)^2\left( E_2\left(\frac{z+1}{t}\right)- \frac{t}{2 \pi i \left({z+1}\right)}  \right)\\
&= \left({z+1}\right)^2\left( E_2\left({z+1}\right)  - \frac{1}{t} E_2\left(\frac{z+1}{t}\right) \right),
\end{align*}
which proves (\ref{eq:17}).
Similarly, by using $E_2(T^{t}(z))=E_2(z)$ and (\ref{eq:12}) we find
\begin{align}
E_{2}(A_{p_1}^{-1}z)&=E_{2}(S T^{-p_1} S z)\nonumber\\
&=\left(\frac{-p_1 z -1 }{z}\right)^2 \left(E_{2}(S z)-\frac{z}{2 \pi i \left({-p_1 z -1 }\right)} \right)\nonumber\\
&=\left({p_1 z +1 }\right)^2 \left(  E_2(z)- \frac{p_1}{2 \pi i \left({p_1 z +1 }\right)} \right),\label{eq:20}
\end{align}
and by $E_2(S^{2}(z))=E_2(z)$, $E_2(T^{t}(z))=E_2(z)$ and (\ref{eq:12}) we find
\begin{align}
E_{(2,p_1p_2)}(A_{p_1}^{-1}z)&=E_{2}\left( S^2 T^{-p_2} S \left( \frac{p_1z+1}{p_2} \right) \right)\nonumber \\
&=E_{2}\left( S \left( \frac{p_1z+1}{p_2} \right) \right) \nonumber \\
&=\left( \frac{p_1z+1}{p_2} \right)^2 \left( E_{2} \left( \frac{p_1z+1}{p_2} \right) - \frac{p_2}{2 \pi i \left({p_1z+1} \right)} \right) \label{eq:21}.
\end{align}
Combining (\ref{eq:20}) and (\ref{eq:21}) we have
\begin{align*}
L_{p_1p_2}(A_{p_1}^{-1}z)&=E_{2}(A_{p_1}^{-1}z)- p_1p_2 E_{(2,p_1p_2)}(A_{p_1}^{-1}z)\\
&= \left({p_1 z +1 }\right)^2 \left(  E_2(z)- \frac{p_1}{2 \pi i \left({p_1 z +1 }\right)} \right)\\
&- p_1p_2 \left(\left( \frac{p_1z+1}{p_2} \right)^2 \left( E_{2} \left( \frac{p_1z+1}{p_2} \right) - \frac{p_2}{2 \pi i\left({p_1z+1} \right)} \right)\right)\\
&= \left({p_1 z +1 }\right)^2 \left(  E_2(z) -\frac{p_1}{p_2} E_{2} \left(\frac{p_1z+1}{p_2} \right) \right),
\end{align*}
which proves (\ref{eq:18}). The proof of (\ref{eq:19}) is similar.
\end{proof}
The following theorem is a special case of Theorem 4.1 from \cite{aaarmf}.
\begin{theorem} \label{th:3}
Let $t, c \in \nn$ be such that $c \mid t$. Then Fourier series expansion of $E_{(4,t)}(z)$ at the cusp $1/c$ is given by
\beqars
E_{(4,t)}(A_c^{-1}z)=\left(\frac{c}{t}\right)^4 (cz+1)^4 E_4\left(\frac{c^2 z+c}{t}\right).
\eeqars
\end{theorem}
\begin{proof}
Let $t, c \in \nn$ be such that $c \mid t$. Let $L=\left(
\begin{array}{cc}
-{t/c} & 1 \\
-1 & 0
\end{array}
\right) \in \Gamma_0(1)$ and $\gamma=\left(
\begin{array}{cc}
-t & 0 \\
-c & -1
\end{array}
\right)$ be matrices. Then since $E_4(z) \in M_4(\Gamma_0(1))$, we have
\beqars
E_{(4,t)}(A_c^{-1}z)=E_4(\gamma z)=E_4\left(L\left(\frac{c^2z+c}{t} \right)  \right)=\left(\frac{c^2z+c}{t} \right)^4 E_4\left(\frac{c^2z+c}{t}  \right),
\eeqars
which proves the assertion.
\end{proof}

The following cusps in Table \ref{table:4} are equivalent in the given modular subgroups.
\begin{center}
	\begin{longtable}{l | l}
		\caption{Equivalence of certain cusps} \label{table:4}
		\endhead
		 Modular Subgroups & Equivalent cusps   \\
		\hline \\[-2mm]
		$\Gamma_0(1)$ & $1 \sim 1/{p_1}\sim 1/{p_2} \sim \infty$\\
	    $ \Gamma_0(p_1) $  & $1 \sim 1/{p_2}$ and $1/{p_1} \sim \infty$  \\
		$ \Gamma_0(p_2) $  & $1 \sim 1/{p_1}$ and $1/{p_2} \sim \infty$ 
	\end{longtable}
\end{center}

We construct Tables \ref{table:1} --\ref{table:3} below by using (\ref{eq:5}), Theorems \ref{th:2}, \ref{th:3}, the definition of cusp forms and equivalence of the cusps given in Table \ref{table:4}. Tables \ref{table:1}, \ref{table:2} and \ref{table:3} will be used to prove (\ref{eq:1}), (\ref{eq:2}) and (\ref{eq:16}), respectively.

\begin{center}
\begin{longtable}{l | c  c c c}
\caption{First terms of modular forms at certain cusps} \label{table:1}
\endhead
cusp $1/c$ & {\rm $[0]_c(L_p(z))^2$} & {\rm $[0]_cE_4(z)$} & {\rm $[0]_cE_4(pz)$} & {\rm $[0]_cC_p(z)$}  \\
\hline \\[-2mm]
$c=0$ & $ (1-p)^2 $  & $ 1 $ & $ 1 $  & $ 0 $  \\
$c=1$  & $ \left(\frac{1-p}{p} \right)^2$ & $ 1 $ & $ \frac{1}{p^4} $ &  $ 0 $ \\
\end{longtable}
\end{center}

{
\small
\begin{longtable}{l | c  c c c c c}
\caption{First  terms of modular forms at certain cusps} \label{table:2}
\endhead
{\rm cusp $1/c$} &  $[0]_c L_{p_1}(z)L_{p_2}(z)$ & $[0]_c E_4(z) $  & $[0]_c E_4(p_1z) $  & $[0]_c E_4(p_2z) $  & $[0]_c E_4(p_1p_2z) $ & $[0]_c C_{(p_1,p_2)}(z) $  \\
\hline \\
$c=0$  & $ (1-p_1)(1-p_2) $  &  $ 1 $ & $ 1 $  & $ 1 $ &  $1  $ &  $0 $  \\ 
$c=1$  & $\ds \frac{(1-p_1)(1-p_2)}{p_1p_2} $ & $1 $  & $ \ds \frac{1}{p_1^4}$ & $ \ds \frac{1}{p_2^4} $ & $ \ds \frac{1}{p_1^4p_2^4} $ &  $0 $ \\ 
 $\ds c={p_1}$ & $ \ds \frac{(1-p_1)(p_2-1)}{p_2} $ &  $ 1 $ & $ 1 $  & $\ds \frac{1}{p_2^4}$  & $\ds \frac{1}{p_2^4}  $&  $0 $\\ 
 $\ds c={p_2}$ & $ \ds \frac{(p_1-1)(1-p_2)}{p_1} $ & $ 1$  & $ \ds \frac{1}{p_1^4}$ & $1 $ & $\ds \frac{1}{p_1^4} $ &  $0 $\\ 
 \end{longtable}
}
{
	\small
	\begin{longtable}{l | c  c c c c c}
		\caption{First  terms of modular forms at certain cusps} \label{table:3}
		\endhead
		{\rm cusp $1/c$} &  $[0]_c (L_{p_1p_2}(z))^2$ & $[0]_c E_4(z) $  & $[0]_c E_4(p_1z) $  & $[0]_c E_4(p_2z) $  & $[0]_c E_4(p_1p_2z) $ & $[0]_c C_{(1,p_1p_2)}(z) $  \\
		\hline \\
		$c=0$  & $ (1-p_1p_2)^2 $  &  $ 1 $ & $ 1 $  & $ 1 $ &  $1  $ &  $0 $  \\ 
		$c=1$  & $\ds \left( \frac{1-p_1p_2}{p_1 p_2}\right)^2 $ & $1 $  & $ \ds \frac{1}{p_1^4}$ & $ \ds \frac{1}{p_2^4} $ & $ \ds \frac{1}{p_1^4p_2^4} $ &  $0 $ \\ 
		$\ds c={p_1}$ & $ \ds \left( 1-\frac{p_1}{p_2}\right)^2 $ &  $ 1 $ & $ 1 $  & $\ds \frac{1}{p_2^4}$  & $\ds \frac{1}{p_2^4}  $&  $0 $\\ 
		$\ds c={p_2}$ & $ \ds \left( 1-\frac{p_2}{p_1}\right)^2 $ & $ 1$  & $ \ds \frac{1}{p_1^4}$ & $1 $ & $\ds \frac{1}{p_1^4} $ &  $0 $\\ 
	\end{longtable}
}

\section{Proof of Theorem \ref{th:1} } \label{sec:4}
Let $p,p_1,p_2$ be primes such that $p_1 \neq p_2$. By (\ref{eq:3}), we have $(L_p(z))^2 \in M_4(\Gamma_0(p))$, and $L_{p_1}(z)L_{p_2}(z),$ $(L_{p_1p_2}(z))^2 \in M_4(\Gamma_0(p_1p_2))$. By \cite[Theorem 5.9]{stein}, the sets of Eisenstein series 
\beqars
&& \{ E_4(z), E_4(pz) \}, \\
&& \{ E_4(z), E_4(p_1z), E_4(p_2z), E_4(p_1p_2z) \}
\eeqars
constitute bases for $E_4(\Gamma_0(p))$ and $E_4(\Gamma_0(p_1p_2))$, respectively. Then by (\ref{eq:10}), we obtain
\beqar
&& (L_p(z))^2=a_1 E_4(z)+a_2 E_4(pz)+C_{(1,p)}(z), \label{eq:7}\\
&& L_{p_1}(z)L_{p_2}(z)=b_1 E_4(z)+b_2 E_4(p_1z)+b_3 E_4(p_2z)+b_4 E_4(p_1p_2z)+C_{(p_1,p_2)}(z),\label{eq:22} \\
&& (L_{p_1p_2}(z))^2=b_1 E_4(z)+b_2 E_4(p_1z)+b_3 E_4(p_2z)+b_4 E_4(p_1p_2z)+C_{(1,p_1p_2)}(z),\label{eq:23}
\eeqar
for some $a_i, b_j \in \cc$, $C_{(1,p)}(z) \in S_4(\Gamma_0(p))$, $C_{(p_1,p_2)}(z),$ and $C_{(1,p_1p_2)}(z) \in S_4(\Gamma_0(p_1p_2))$.

We use the values in Table \ref{table:1} to compare first terms of Fourier series expansions of the functions on both sides of equation (\ref{eq:7}) to obtain the following linear equations.
\beqar
&& (1-p)^2=a_1+a_2, \label{eq:8}\\
&& \left(\frac{p-1}{p} \right)^2= a_1 + a_2 \frac{1}{p^4} \label{eq:9}.
\eeqar
Solving equations (\ref{eq:8}) and (\ref{eq:9}) for $a_1$ and $a_2$ we get the desired result in (\ref{eq:1}).

We similarly prove (\ref{eq:2}) and (\ref{eq:16}) using the values in Table \ref{table:2} in (\ref{eq:22}) and Table \ref{table:3} in (\ref{eq:23}), respectively.

\section{The cusp forms $C_{(r,s)}(z)$ for $(r,s)=(1,2),$ $(1,3),$ $(1,5),$ $(1,7),$ $(1,11),$ $(2,3),$ $(2,5),$ $(2,7),$ $(3,5),$ $(1,6),$ $(1,10),$ $(1,14),$ $(1,15)$} \label{sec:5}
In this section we express $C_{(r,s)}(z)$ $(r,s)=(1,2),$ $(1,3),$ $(1,5),$ $(1,7),$ $(1,11),$ $(2,3),$ $(2,5),$ $(2,7),$ $(3,5),$ $(1,6),$ $(1,10),$ $(1,14),$ $(1,15)$ as linear combinations of eta quotients. The Dedekind eta function $\eta (z)$ is the holomorphic function defined on the upper half plane $\hh$ 
by the product formula
\beqar
\eta (z) = e^{\pi i z/12} \prod_{n=1}^{\infty} (1-e^{2\pi inz}).
\eeqar
Let $N \in \nn$, an eta quotient (of level $N$) is defined to be a finite product of the form
\beqar
f(z) = \prod_{\delta  \mid N } \eta^{r_{\delta}} ( \delta z),
\eeqar
where $\delta$ runs through positive divisors of $N$ and $r_{\delta} \in \zz $, not all zeroes. For convenience we use the notation
\beqars
\prod_{\delta  \mid N } \eta^{r_{\delta}} ( \delta z)=\eta[1^{r_1},\ldots,\delta^{r_\delta},\ldots,N^{r_N}](z),
\eeqars
for an eta quotient.
\begin{theorem}\label{th:4} We express $C_{(r,s)}(z)$ $((r,s)=(,)$ in terms of eta quotients as follows.
\beqars
& C_{(1,2)}(z)=&0,\\
& C_{(1,3)}(z)=&0,\\
& C_{(1,5)}(z)=&\frac{576}{13} \eta[1^4, 5^4](z),\\
& C_{(1,7)}(z)= & \frac{576}{5}\eta[1^5, 2^{-1}, 7^5, 14^{-1}](z)+\frac{2304}{5} \eta[1^2, 2^2, 7^2, 14^2](z),\\
& C_{(1,11)}(z)=& \frac{17280}{61} \eta[1^6, 2^{-2}, 11^6, 22^{-2}](z)  +\frac{118656}{61} \eta[1^4, 11^4](z) \\
&&+ \frac{276480}{61} \eta[1^2, 2^2, 11^2, 22^2](z) +\frac{276480}{61}   \eta[ 2^4, 22^4](z),\\
& C_{(2,3)}(z)=&-\frac{144}{5}\eta[1^2, 2^2, 3^2, 6^2](z),\\
& C_{(2,5)}(z)=&\frac{48}{13}\eta[1^{-2}, 2^8, 5^2](z)-\frac{6000}{13} \eta[1^2, 5^{-2}, 10^8](z),\\
& C_{(2,7)}(z)=&\frac{288}{5}\eta[1^5, 2^{-1}, 7^5, 14^{-1}](z) -\frac{336}{25} \eta[1^{-2}, 2^6, 7^6, 14^{-2}](z), \\
&&+\frac{8064}{25} \eta[1^2, 2^2, 7^2, 14^2](z) -\frac{5136}{25} \eta[1^6, 2^{-2}, 7^{-2}, 14^6](z),\\
& C_{(3, 5)}(z)=& \frac {108}{13}\eta[1^{11}, 2^{-5}, 3^{-5}, 5^{-5}, 6^{3}, 10^3, 15^{11}, 30^{-5}](z) \\
&&-\frac {1584}{13} \eta[[1^{2}, 2^{ -3}, 3^{ -4}, 5^{ -3}, 6^{ 9}, 10^{ 4}, 15^{ 7}, 30^{ -4}](z) \\
&&-\frac {2376}{13}\eta[1^{4}, 2^{ -3}, 3^{ -2}, 5^{ 2}, 6^{ 5}, 10^{ 1}, 15^{ 0}, 30^{ 1}](z)   \\
&&-\frac {5832}{13}\eta[1^{1}, 2^{ -2}, 3^{ -3}, 5^{ 2}, 6^{ 8}, 10^{ -1}, 15^{ 2}, 30^{ 1}](z)  \\
&& +\frac {11448}{13}\eta[1^{2}, 2^{ -2}, 3^{ 0}, 5^{ -2}, 6^{ 2}, 10^{ 6}, 15^{ 0}, 30^{ 2}](z)   \\
&&+\frac {33264}{13}\eta[1^{1}, 2^{ 2}, 3^{ 4}, 5^{ 3}, 6^{ -3}, 10^{ -2}, 15^{ -4}, 30^{ 7}](z)   \\
&&+\frac {10080}{13}\eta[1^{1}, 2^{ 0}, 3^{ -1}, 5^{ -3}, 6^{ 0}, 10^{ 2}, 15^{ 7}, 30^{ 2}](z)   \\
&&+\frac {83520}{13}\eta[1^{1}, 2^{ -1}, 3^{ -2}, 5^{ 1}, 6^{ 4}, 10^{ -1}, 15^{ 0}, 30^{ 6}](z)   \\
&&+\frac {145584}{13}\eta[1^{0}, 2^{ 1}, 3^{ 0}, 5^{ 2}, 6^{ -1}, 10^{ -3}, 15^{ 2}, 30^{ 7}](z) \\
&&-\frac {117720}{13}\eta[1^{-2}, 2^{ 3}, 3^{ 5}, 5^{ 3}, 6^{ -4}, 10^{ -4}, 15^{ -4}, 30^{ 11}](z)   \\
&&-1944 \eta[1^{1}, 2^{ 0}, 3^{ 1}, 5^{ 3}, 6^{ 0}, 10^{ -4}, 15^{ -5}, 30^{ 12}](z)   \\
&&+\frac {158688}{13} \eta[1^{0}, 2^{ 1}, 3^{ 1}, 5^{ 2}, 6^{ -2}, 10^{ -3}, 15^{ -3}, 30^{ 12}](z)  \\
&&+\frac {216000}{13} \eta[1^{-1}, 2^{ 3}, 3^{ 1}, 5^{ 3}, 6^{ -1}, 10^{ -5}, 15^{ -7}, 30^{ 15}](z)   \\
&&+\frac {316224}{13} \eta[1^{0}, 2^{ 1}, 3^{ 2}, 5^{ 2}, 6^{ -3}, 10^{ -3}, 15^{ -8}, 30^{ 17}](z), \\
& C_{(1,6)}(z)=& \frac{288}{5}\eta[1^2, 2^2, 3^2, 6^2](z), \\
& C_{(1,10)}(z)=& \frac{2976}{13}\eta[1^{-1}, 2^5, 5^5, 10^{-1}](z)-\frac{480}{13} \eta[1^5, 2^{-1}, 5^{-1}, 10^5](z),\\
& C_{(1,14)}(z)=& \frac{10272}{25} \eta[1^{-2}, 2^6, 7^6, 14^{-2}](z) -\frac{4608}{25} \eta[1^2, 2^2, 7^2, 14^2](z) \\
&& +\frac{672}{25} \eta[1^6, 2^{-2}, 7^{-2}, 14^6](z),\\
& C_{(1,15)}(z)=& \frac {5976}{13}\eta[1^{11}, 2^{-5}, 3^{-5}, 5^{-5}, 6^{3}, 10^3, 15^{11}, 30^{-5}](z) \\
&&+\frac {74592}{13} \eta[[1^{2}, 2^{ -3}, 3^{ -4}, 5^{ -3}, 6^{ 9}, 10^{ 4}, 15^{ 7}, 30^{ -4}](z) \\
&&-\frac {137808}{13}\eta[1^{4}, 2^{ -3}, 3^{ -2}, 5^{ 2}, 6^{ 5}, 10^{ 1}, 15^{ 0}, 30^{ 1}](z)   \\
&&-\frac {85968}{13}\eta[1^{1}, 2^{ -2}, 3^{ -3}, 5^{ 2}, 6^{ 8}, 10^{ -1}, 15^{ 2}, 30^{ 1}](z) \\
&& +\frac {153072}{13}\eta[1^{2}, 2^{ -2}, 3^{ 0}, 5^{ -2}, 6^{ 2}, 10^{ 6}, 15^{ 0}, 30^{ 2}](z)   \\
&& -\frac {617184}{13}\eta[1^{1}, 2^{ 2}, 3^{ 4}, 5^{ 3}, 6^{ -3}, 10^{ -2}, 15^{ -4}, 30^{ 7}](z)   \\
&& +\frac {1513728}{13}\eta[1^{1}, 2^{ 0}, 3^{ -1}, 5^{ -3}, 6^{ 0}, 10^{ 2}, 15^{ 7}, 30^{ 2}](z)   \\
&& -\frac {2807424}{13}\eta[1^{1}, 2^{ -1}, 3^{ -2}, 5^{ 1}, 6^{ 4}, 10^{ -1}, 15^{ 0}, 30^{ 6}](z)   \\
&& -\frac {943200}{13}\eta[1^{0}, 2^{ 1}, 3^{ 0}, 5^{ 2}, 6^{ -1}, 10^{ -3}, 15^{ 2}, 30^{ 7}](z)   \\
&& +135504 \eta[1^{-2}, 2^{ 3}, 3^{ 5}, 5^{ 3}, 6^{ -4}, 10^{ -4}, 15^{ -4}, 30^{ 11}](z)   \\
&& +\frac {3255696}{13}\eta[1^{1}, 2^{ 0}, 3^{ 1}, 5^{ 3}, 6^{ 0}, 10^{ -4}, 15^{ -5}, 30^{ 12}](z)   \\
&& -\frac {3409344}{13}\eta[1^{0}, 2^{ 1}, 3^{ 1}, 5^{ 2}, 6^{ -2}, 10^{ -3}, 15^{ -3}, 30^{ 12}](z)  \\
&& +\frac {586368}{13}\eta[1^{-1}, 2^{ 3}, 3^{ 1}, 5^{ 3}, 6^{ -1}, 10^{ -5}, 15^{ -7}, 30^{ 15}](z)   \\
&& -\frac {10461312}{13}\eta[1^{0}, 2^{ 1}, 3^{ 2}, 5^{ 2}, 6^{ -3}, 10^{ -3}, 15^{ -8}, 30^{ 17}](z).
\eeqars
\end{theorem}
\begin{proof} Proofs of all these equalities are similar. We basically find combinations of eta quotients whose first couple terms in Fourier series expansion at $\infty$ agrees with the first couple terms in Fourier series expansion of $C_{(r,s)}(z)$ at $\infty$. Then by Sturm Theorem these two modular forms are equal, see  (\cite[Theorem 3.13]{Kilford}) for Sturm Theorem (or Sturm Bound).
To illustrate the proof we show
\beqar
& C_{(1,11)}(z)=& \frac{17280}{61} \eta[1^6, 2^{-2}, 11^6, 22^{-2}](z)  +\frac{118656}{61} \eta[1^4, 11^4](z) \label{eq:26} \\
&&+ \frac{276480}{61} \eta[1^2, 2^2, 11^2, 22^2](z) +\frac{276480}{61}   \eta[ 2^4, 22^4](z).\nonumber
\eeqar
By Theorem \ref{th:1} we have
\begin{align*}
\ds & C_{(1,11)}(z)= (L_{11}(z))^2-\frac{50}{61}E_4(z)-\frac{6050}{61}E_4(11z) \in S_4(\Gamma_0(11)) \subset S_4(\Gamma_0(22)).
\end{align*}
We use MAPLE to expand $(L_{11}(z))^2$, $E_4(z)$ and $E_4(11z)$. Then we find
\begin{align}
\ds C_{(1,11)}(z)=& {\frac {17280}{61}}q+{\frac {14976}{61}}{q}^{2}-{\frac {8064}{61}}{q}
^{3}-{\frac {73728}{61}}{q}^{4}-{\frac {1152}{61}}{q}^{5}-{\frac {
213120}{61}}{q}^{6} \label{eq:24}\\
&+{\frac {182016}{61}}{q}^{7}-{\frac {80640}{61}}{q}
^{8}+{\frac {361728}{61}}{q}^{9}+{\frac {411264}{61}}{q}^{10}-{\frac {
190080}{61}}{q}^{11}\nonumber\\
&-{\frac {377856}{61}}{q}^{12}+O(q^{13}). \nonumber
\end{align}
On the other hand, let
\begin{align*}
f(z)= & \frac{17280}{61} \eta[1^6, 2^{-2}, 11^6, 22^{-2}](z)  +\frac{118656}{61} \eta[1^4, 11^4](z)+ \frac{276480}{61} \eta[1^2, 2^2, 11^2, 22^2](z) \\
&+\frac{276480}{61}   \eta[ 2^4, 22^4](z).
\end{align*}
We have $f(z) \in S_4(\Gamma_0(22))$, see  \cite[Theorem 5.7, p. 99]{Kilford}, \cite[Corollary 2.3, p. 37]{Kohler}, 
\cite[p. 174]{GordonSinor} and \cite{Ligozat}. Again using MAPLE we also compute 
\begin{align}
 f(z)=&{\frac {17280}{61}}q+{\frac {14976}{61}}{q}^{2}-{\frac {8064}{61}}{q}
^{3}-{\frac {73728}{61}}{q}^{4}-{\frac {1152}{61}}{q}^{5}-{\frac {
213120}{61}}{q}^{6} \label{eq:25}\\
&~+{\frac {182016}{61}}{q}^{7}-{\frac {80640}{61}}{q}
^{8}+{\frac {361728}{61}}{q}^{9}+{\frac {411264}{61}}{q}^{10}-{\frac {
190080}{61}}{q}^{11} \nonumber\\
&~ -{\frac {377856}{61}}{q}^{12}+O(q^{13}). \nonumber
\end{align}
That is, by (\ref{eq:24}) and (\ref{eq:25}), we have
\beqars
f(z)-C_{(1,11)}(z)=O(q^{13}).
\eeqars
Thus by Sturm Theorem (see \cite[Theorem 3.13]{Kilford}), we have
\begin{align*}
\ds f(z)-C_{(1,11)}(z)=0,
\end{align*}
which proves (\ref{eq:26}).
\end{proof}
\section{Further extensions} \label{sec:6} 
Let $N_1 \in \nn$ be a square-free number and $N=N_1,2N_1,$ or $4N_1$. Let $r,s\in \nn$ be such that $\lcm(r,s) \mid N$. We can use similar arguments to give formulas for $W(r,s;n)$. We consider $L_r(z) L_s(z) \in M_4(\Gamma_0(N))$. Then by \cite[Theorem 5.9]{stein} and (\ref{eq:10}), we have that there exists $c_d \in \cc$ and $C_{(r,s)}(z) \in S_4(\Gamma_0(N))$, such that
\beqar
L_r(z)L_s(z)=\sum_{d \mid N} c_d E_4(dz) + C_{(r,s)}(z). \label{eq:11}
\eeqar
Now in order to compute $c_d$ we need to find the first terms of the modular forms on both sides of equation (\ref{eq:11}) at the cusps of $\Gamma_0(N)$. Note that, because of the choice of $N$, the set
\beqars
\left\{ \frac{1}{c} : c \mid N \right\}
\eeqars
is a set of cusps of $\Gamma_0(N)$. Then comparing first coefficients of Fourier series expansions of modular forms in (\ref{eq:11}), we will find a set of linear equations. We can determine $c_d$ by solving these equations. Also note that in many cases $C_{(r,s)}(z)$ can be given in terms of eta quotients. Finally a formula for $W(r,s;n)$ can be given by comparing coefficients of $q^n$ on both sides of the equation (\ref{eq:11}). 

\section*{Acknowledgments}  
The author was partially supported by the Singapore Ministry of Education Academic Research Fund, Tier 2, project number MOE2014-T2-1-051, ARC40/14.

\noindent
Zafer Selcuk Aygin\\
Division of Mathematical Sciences \\
School of Physical and Mathematical Sciences \\
Nanyang Technological University \\
21 Nanyang Link, Singapore 637371, Singapore

\vspace{1mm}

\noindent
selcukaygin@ntu.edu.sg


\begin{thebibliography}{0}
\bibitem{12glaisher}
{J. W. L. Glaisher,}
\newblock {\it On the square of the series in which the coefficients are the sums of the divisors of the exponents,}
\newblock Mess. Math. 14 (1885) 156--163.
\end{thebibliography}
\end{document}